\documentclass[a4paper,10.5pt]{article}
\usepackage{amsmath,amssymb,amsthm}

\usepackage{color}
\theoremstyle{definition}
 \newtheorem{dfn}{Definition}[section]
 \newtheorem{remark}[dfn]{Remark}

\theoremstyle{plain}
 \newtheorem{thm}[dfn]{Theorem}
   
 \newtheorem{lem}[dfn]{Lemma}

\numberwithin{equation}{section}

\newcommand{\bD}{{\bold D}}

\newcommand{\bG}{{\bold G}}
\newcommand{\bH}{{\bold H}}
\newcommand{\bI}{{\mathbb I}}

\newcommand{\bK}{{\bold K}}
\newcommand{\bL}{{\bold L}}
\newcommand{\bM}{{M}}

\newcommand{\bS}{{\bold S}}
\newcommand{\bT}{{\bold T}}

\newcommand{\DV}{{\rm Div}\,}
\newcommand{\dv}{\, {\rm div}\,}

\newcommand{\BR}{{\Bbb R}}
\newcommand{\BC}{{\Bbb C}}

\newcommand{\BN}{{\Bbb N}}

\newcommand{\CF}{{\mathcal F}}
\newcommand{\CI}{{\mathcal I}}

\newcommand{\CL}{{\mathcal L}}

\newcommand{\CN}{{\mathcal N}}

\newcommand{\ba}{{\bold a}}
\newcommand{\bb}{{\bold b}}

\newcommand{\bff}{{\bold f}}
\newcommand{\bv}{{\bold v}}
\newcommand{\bu}{{\bold u}}
\newcommand{\bw}{{\bold w}}
\newcommand{\bg}{{\bold g}}
\newcommand{\bm}{{\bold m}}

\newcommand{\pd}{\partial}

\newcommand{\R}{\mathbb{R}}

\title{The global well-posedness of the compressible 
fluid model of Korteweg type for the critical case}
\author{
Takayuki KOBAYASHI
\thanks{Division of Mathematical Science, 
Department of Systems Innovation, 
Graduate School of Engineering Science, \endgraf
Osaka University,
1-3, Machikaneyama-cho, Toyonaka-shi, Osaka,
560-8531, Japan. \endgraf
e-mail address: kobayashi@sigmath.es.osaka-u.ac.jp}
\enskip and \enskip
Miho MURATA
\thanks{
Department of Mathematical and System Engineering,
Faculty of Engineering,
Shizuoka University, \endgraf
3-5-1 Johoku, Naka-ku, Hamamatsu-shi, Shizuoka,
432-8561, Japan.
\endgraf
e-mail address: murata.miho@shizuoka.ac.jp}
}
\date{}

\begin{document}
\maketitle

\begin{abstract}
In this paper, we consider the compressible 
fluid model of Korteweg type
in a critical case where the derivative of pressure equals to $0$ at the given
constant state.
It is shown that the system admits a unique, global strong solution 
for small initial data in the maximal $L_p$-$L_q$ regularity class.
As a result, we also prove the decay estimates of the solutions to the nonliner problem.
In order to obtain the global well-posedness for the critical case, we show 
$L_p$-$L_q$ decay properties of solutions to the linearized equations 
under an additional assumption for a low frequencies.
\end{abstract}

\section{Introduction}
We consider the following compressible viscous 
fluid model of Korteweg type in the $N$ dimensional 
Euclidean space $\BR^N$, $3\leq N \leq 7$. 

\begin{equation}\label{nsk}
\left\{
\begin{aligned}
&\pd_t \rho + \dv (\rho \bu) = 0 &\quad&\text{in $\R^N$ for $t \in (0, T)$}, \\
&\rho (\pd_t \bu + \bu \cdot \nabla \bu)  
- \DV \bT + \nabla P(\rho) =0 & \quad&\text{in $\R^N$ for $t \in (0, T)$}, \\
&(\rho, \bu)|_{t=0} = (\rho_* + \rho_0, \bu_0) &\quad&\text{in $\R^N$},
\end{aligned}
\right.
\end{equation}
where $\pd_t = \pd/\pd t$, $t$ is the time variable, 
$\rho = \rho(x, t)$, $x=(x_1, \ldots, x_N) \in \BR^N$
and  
$\bu = \bu(x, t) = (u_1(x, t), \ldots, u_N(x, t))$
are respective unknown density field and velocity field, 
$P(\rho)$ is the pressure field 
satisfying a $C^\infty$ function defined on
$\rho > 0$, 
where $\rho_*$ is a positive constant.
Moreover, $\bT = \bS (\bu) + \bK (\rho)$
is the stress tensor, where $\bS(\bu)$ and 
$\bK(\rho)$ are respective the viscous stress tensor and 
Korteweg stress tensor given by 
\begin{align*}
\bS (\bu) &= \mu_* \bD(\bu) + 
(\nu_* - \mu_*) \dv \bu \bI, \\
\bK (\rho) &= \frac{\kappa_*}{2} (\Delta \rho^2 -  |\nabla \rho|^2 )\bI 
- \kappa_* \nabla \rho \otimes \nabla \rho.
\end{align*} 
Here, 
$\bD(\bu)$ denotes 
the deformation tensor whose $(j, k)$ components are
 $D_{jk}(\bu) = \pd_ju_k
+ \pd_ku_j$ with $\pd_j
= \pd/\pd x_j$.
For any vector of functions $\bv = (v_1, \ldots, v_N)$, 
we set $\dv \bv = \sum_{j=1}^N\pd_jv_j$,
and also for any $N\times N$ matrix field $\bL$ with $(j,k)^{\rm th}$ components $L_{jk}$, 
the quantity $\DV \bL$ is an 
$N$-vector with $j^{\rm th}$ component $\sum_{k=1}^N\pd_kL_{jk}$.
$\bI$ is the $N\times N$ identity matrix
and
$\ba \otimes \bb$ denotes an $N\times N$ matrix with $(j, k)^{\rm th}$
component $a_j b_k$
for any two $N$-vectors $\ba = (a_1, \dots, a_N)$ and $\bb = (b_1, \dots, b_N)$.  
We assume that the viscosity coefficients $\mu_*$, $\nu_*$,
 the capillary coefficient $\kappa_*$, and the mass density $\rho_*$ of the
 reference body satisfy
the conditions:

\begin{equation}\label{condi} 
\mu_* > 0, \quad \mu_* + \nu_*>0, \quad \text{and} \quad \kappa_* > 0.
\end{equation}
Furthermore, we assume that the pressure $P(\rho)$ satisfies 
\begin{equation}\label{pressure} 
 P'(\rho_*) = 0.
\end{equation}

The system \eqref{nsk} governs 
the motion of the compressible fluids with 
capillarity effects, which was proposed by 
Korteweg \cite{K} as a diffuse interface model for liquid-vapor flows
based on Van der Waals's approach \cite{Wa}
and derived rigorously by Dunn and Serrin in \cite{DS}.
As shown in \cite{D}, 
since the Korteweg model was drived by using on Van der
Waals potential,
the pressure is non-monotone in general.
This is one of the important aspects to the diffuse interface model,
so that we consider the system \eqref{nsk} under the condition \eqref{pressure}.

There are many mathematical results on global solutions 
of Korteweg model in the case $P'(\rho) > 0$.
Bresch, Desjardins, and Lin \cite{BDL} 
proved the existence of global weak solution,
 and then Haspot improved their result in \cite{H}.
Hattori and Li \cite{HL1, HL2} first showed
the local and global unique existence 
in Sobolev space. 
They assumed the initial data 
$(\rho_0, \bu_0)$ belong to $H^{s + 1} (\BR^N) \times H^s (\BR^N)^N$
$(s \geq [N/2] + 3)$.
Hou, Peng, and Zhu \cite{HPZ} improved the results \cite{HL1, HL2}
when the total energy is small. 
Wang and Tan \cite{WT}, 
Tan and Wang \cite{TW}, 
Tan, Wang, and Xu \cite{TWX}, and 
Tan and Zhang \cite{TZ} 
established the optimal decay rates of the global solutions
in Sobolev space.
Li \cite{L} and Chen and Zhao \cite{CZ} 
considered Navier-Stokes-Korteweg system with external force.
Bian, Yeo, and Zhu \cite{BYZ} obtained 
the vanishing capillarity limit of the smooth solution.
We also refer to the existence and uniqueness results 
in critical Besov space proved by Danchin and Desjardins in \cite{DD}.
Their initial data $(\rho_0, \bu_0)$ are assumed to belong to
$\dot{B}^{N/2}_{2,1}(\BR^N) \cap \dot{B}^{N/2-1}_{2,1}(\BR^N)
 \times \dot{B}^{N/2-1}_{2,1}(\BR^N)^N$.
Recently, Murata and Shibata \cite{MS} proved the global well-posedness
in the maximal $L_p$-$L_q$ regularity class.

On the other hand, there are few results in the case where $P'(\rho_*) = 0$.
Kobayashi and Tsuda \cite{KT} proved the existence of global $L_2$ solutions 
and the decay estimates.
Chikami and Kobayashi \cite{CK} improved the result \cite{DD}.
In particular, when $P'(\rho_*) = 0$, 
they proved the global estimates
under an additional low frequency assumption
to control a pressure term.
Furthermore, they showed the optimal decay rates
of the global solutions
in the $L_2$-framework.
Recently, Watanabe \cite{W} proved 
the global well-posedness in the maximal $L_p$-$L_q$ regularity class including the $L_2$-framework
under the condition \eqref{condi} and the additional assumption $(\mu_* + \nu_*)^2/\rho_*^2 \geq 4 \rho_* \kappa_*$.
It is not clear about the decay estimates of the solutions to the nonlinear problem \eqref{nsk}.

In this paper, we discuss the global existence and uniqueness
of the strong solutions to \eqref{nsk}
for small initial data
under the assumption \eqref{pressure}.
As a result, we also prove the decay estimates of the solutions to \eqref{nsk}.
The main tools are the maximal $L_p$-$L_q$ regularity
and $L_p$-$L_q$ decay properties of the linearized equations. 
In order to consider the linearized problem,
we first rewrite \eqref{nsk} to the momentum formulation,
which helps our analysis when $P'(\rho_*) = 0$. 
In fact, we assume that the initial data for the momentum has divergence form
in order to obtain the suitable decay
properties of  the low frequency part of solutions to the linearized equations
as shown in Theorem \ref{semi}, below. 
Since the nonlinear terms of the momentum formulation
is written in divergence form,  
we can use Theorem \ref{semi}.
Moreover, thanks to the condition for the initial data for the momentum,
we do not need the additional assumption for the coefficients as in \cite{W}.

\subsection{Notations}
We summarize several symbols and functional spaces used 
throughout the paper.
$\BN$, $\BR$ and $\BC$ denote the sets of 
all natural numbers, real numbers and complex numbers, respectively. 
We set $\BN_0=\BN \cup \{0\}$ and $\BR_+ = (0, \infty)$. 
Let $q'$ be the dual exponent of $q$
defined by $q' = q/(q-1)$
for $1 < q < \infty$. 
For any multi-index $\alpha = (\alpha_1, \ldots, \alpha_N) 
\in \BN_0^N$, we write $|\alpha|=\alpha_1+\cdots+\alpha_N$ 
and $\pd_x^\alpha=\pd_1^{\alpha_1} \cdots \pd_N^{\alpha_N}$ 
with $x = (x_1, \ldots, x_N)$. 
For scalar function $f$ and $N$-vector of functions $\bg$, we set
\begin{gather*}
\nabla f = (\pd_1f,\ldots,\pd_Nf),
\enskip \nabla \bg = (\pd_ig_j \mid i, j = 1,\ldots, N),\\
\nabla^2 f = \{\pd_i \pd_j f \mid i, j = 1,\ldots, N \},
\enskip \nabla^2 \bg = \{\pd_i \pd_j g_k \mid
i, j, k = 1,\ldots,N\},
\end{gather*} 
where $\pd_i = \pd/\pd x_i$. 
For scalar functions,  $f,g$,  and $N$-vectors of functions, 
$\bff$, $\bg$,  we set 
$(f, g)_{\BR^N} = \int_{\BR^N} f g\,dx$, and 
$(\bff,\bg)_{\BR^N} = \int_{\BR^N} \bff\cdot \bg\,dx$, respectively. 
For Banach spaces $X$ and $Y$, $\CL(X,Y)$ denotes the set of 
all bounded linear operators from $X$ into $Y$ and 
$\rm{Hol}\,(U, \CL(X,Y))$ 
 the set of all $\CL(X,Y)$ valued holomorphic 
functions defined on a domain $U$ in $\BC$. 
For any $1 \leq p, q \leq \infty$,
$L_q(\BR^N)$, $W_q^m(\BR^N)$ and $B^s_{q, p}(\BR^N)$ 
denote the usual Lebesgue space, Sobolev space and 
Besov space, 
while $\|\cdot\|_{L_q(\BR^N)}$, $\|\cdot\|_{W_q^m(\BR^N)}$ and 
$\|\cdot\|_{B^s_{q,p}(\BR^N)}$ 
denote their norms, respectively. We set $W^0_q(\BR^N) = L_q(\BR^N)$
and $W^s_q(\BR^N) = B^s_{q,q}(\BR^N)$. 
$C^\infty(\BR^N)$ denotes the set of all $C^\infty$ functions defined on $\BR^N$. 
$L_p((a, b), X)$ and $W_p^m((a, b), X)$ 
denote the usual Lebesgue space and Sobolev space of 
$X$-valued function defined on an interval $(a,b)$, respectively.
The $d$-product space of $X$ is defined by 
$X^d=\{f=(f, \ldots, f_d) \mid f_i \in X \, (i=1,\ldots,d)\}$,
while its norm is denoted by 
$\|\cdot\|_X$ instead of $\|\cdot\|_{X^d}$ for the sake of 
simplicity. 
We set 
\begin{gather*}
W_q^{m,\ell}(\BR^N)=\{(f,\bg) \mid  f \in W_q^m(\BR^N),
\enskip \bg \in W_q^\ell(\BR^N)^N \}, \enskip 
\|(f, \bg)\|_{W^{m, \ell}_q(\BR^N)} = \|f\|_{W^m_q(\BR^N)}
+ \|\bg\|_{W^\ell_q(\BR^N)}.
\end{gather*}
Let $\CF_x= \CF$ and $\CF^{-1}_\xi = \CF^{-1}$ 
denote the Fourier transform and 
the Fourier inverse transform, respectively, which are defined by 
 setting
$$\hat f (\xi)
= \CF_x[f](\xi) = \int_{\BR^N}e^{-ix\cdot\xi}f(x)\,dx, \quad
\CF^{-1}_\xi[g](x) = \frac{1}{(2\pi)^N}\int_{\BR^N}
e^{ix\cdot\xi}g(\xi)\,d\xi. 
$$
 The letter $C$ denotes generic constants and the constant 
$C_{a,b,\ldots}$ depends on $a,b,\ldots$. 
The values of constants $C$ and $C_{a,b,\ldots}$ 
may change from line to line. We use small boldface letters, e.g. $\bu$ to 
denote vector-valued functions and capital boldface letters, e.g. $\bH$
to denote matrix-valued functions, respectively. 
In order to state our main theorem, 
we set a solution space
and several norms:
\begin{align}
D_{q, p}(\BR^N)
& = B^{3-2/p}_{q, p}(\BR^N) \times B^{2(1-1/p)}_{q, p}(\BR^N)^N, \nonumber \\
X_{p, q, t} &= \{(\rho, \bu) 
 \mid \rho \in L_p((0, t), W^3_q(\BR^N)) \cap W^1_p((0, t), W^1_q(\BR^N)), \nonumber \\
& \bu \in 
L_p((0, t), W^2_q(\BR^N)^N) \cap W^1_p((0, t), L_q(\BR^N)^N),
\quad
\rho_*/4 \leq \rho_* + \rho(t, x) \leq 4 \rho_*\}, \nonumber \\
[ U ]_{q, \ell, (a, t)}
&= \sup_{a \leq s \leq t} <s>^\ell \|U (\cdot, s)\|_{L_q(\BR^N)}
\enskip (U = \rho, \bu, (\rho, \bu). a = 0, 2.), 
\nonumber \\ 
[ \nabla U ]_{q, \ell, (a, t)}
&= \sup_{a \leq s \leq t} <s>^\ell \|\nabla U (\cdot, s)\|_{L_q(\BR^N)}
\enskip (U = \rho, (\rho, \bu). a = 0, 2.), 
\nonumber \\ 
[ U ]_{q, \ell, t}
&
=[ U ]_{q, \ell, (0, t)}, \enskip
[ \nabla U ]_{q, \ell, t}
= [ \nabla U ]_{q, \ell, (0, t)}, \nonumber \\ 
\CN (\rho, \bu) (t)
&=\sum^1_{j=0} \sum^2_{i=1}
\{ [ (\nabla^j \rho, \nabla^j \bu) ]_{\infty, \frac{N}{q_1} + \frac{j}{2}, t} \label{N} \\
&+ [ (\nabla^j \rho, \nabla^j \bu) ]_{q_1, \frac{N}{2q_1} + \frac{j}{2}, t}
+ [ (\nabla^j \rho, \nabla^j \bu) ]_{q_2, \frac{N}{2q_2}+1 + \frac{j}{2}, t} \nonumber \\
&+ \|(<s>^{\ell_i}(\rho, \bu)\|_{L_p((0, t), W^{3, 2}_{q_i}(\BR^N))}\nonumber \\
&+ \|<s>^{\ell_i}(\pd_s \rho, \pd_s \bu)\|_{L_p((0, t), W^{1, 0}_{q_i}(\BR^N))} \}, \nonumber 
\end{align}
where  $<s> = (1 + s)$, $\ell_1 = N/2q_1 - \tau$, $\ell_2 = N/2q_2 + 1 - \tau$,
and $\tau$ is given in Theorem \ref{global}, below.

\subsection{Main theorem}

Setting $\bm = \rho \bu$ and $\rho = \rho_* + \theta$,
we can rewrite \eqref{nsk} to the momentum formulation:
\begin{equation}\label{nsk4}\left\{
\begin{aligned}
&\pd_t \theta + \dv \bm = 0 & \quad&\text{in $\R^N$ for $t \in (0, T)$},  \\
&\pd_t \bm - \frac{1}{\rho_*} \DV \bS(\bm) -\kappa_* \rho_* \nabla \Delta \theta
= \bg(\theta, \bm) & \quad&\text{in $\R^N$ for $t \in (0, T)$}, \\
&(\theta, \bm)|_{t=0} = (\rho_0, \bm_0) & \quad&\text{in $\R^N$},
\end{aligned}
\right.
\end {equation}
where
\begin{align*}
\bg (\theta, \bm)  &= 
 - \DV \Bigl[ \left(\frac{1}{\rho_* + \theta} - \frac{1}{\rho_*} \right) \bm \otimes \bm
+ \frac{1}{\rho_*} \bm \otimes \bm \\
&\enskip - \bS \left(\left(\frac{1}{\rho_* + \theta} - \frac{1}{\rho_*}\right) \bm \right)
- \bK(\theta) + \int^1_0 P''(\rho_* + \tau \theta)(1-\tau) \,d\tau \theta^2 \Bigr],\\
\bm_0 &= (\rho_* + \rho_0) \bu_0.
\end{align*}

We now state our main theorem.

\begin{thm}\label{global} Assume that conditions \eqref{condi} and \eqref{pressure} hold
and that $3 \leq N \leq 7$.
Let $q_1$, $q_2$ and $p$ be numbers such that
\begin{equation}\label{condi:pq}
2<p<\infty,
\enskip
q_1<N<q_2, 
\enskip
2 < q_1 \leq 4,
\enskip
\frac{1}{q_1}=\frac{1}{q_2}+\frac{1}{N},
\enskip
\frac{2}{p}+\frac{N}{q_2}<1.
\end{equation}
Let $\tau$ be a number such that
\begin{equation}\label{tau}
\frac{1}{p}<\tau <\frac{N}{q_2}+\frac{1}{p}.
\end{equation}
Then, there exists a small number $\epsilon>0$ such that for any initial data 
$(\rho_0, \bm_0) 
\in \cap^2_{i=1} D_{q_i, p} (\BR^N) \cap W^{1, 0}_{q_1/2}(\BR^N)$ 
satisfying
\[
\CI :=
\sum^2_{i=1}\|(\rho_0, \bm_0)\|_{D_{q_i, p} (\BR^N)}
+\|(\rho_0, \bm_0)\|_{W^{1, 0}_{q_1/2}(\BR^N)}
+\|(\rho, \bM_0)\|_{L_{q_1/2}(\BR^N)} < \epsilon
\]
with $\bm_0 = \DV \bM_0$,
problem \eqref{nsk4} 
admits a solution $(\theta, \bm)$ with
\[
(\theta, \bm)\in X_{p, q_2, \infty}
\]
satisfying the estimate
\[
\CN(\theta, \bm)(\infty)\leq L \epsilon
\]
with some constant $L$ independent of $\epsilon$.

\end{thm}

\begin{remark} \thetag1~ In Theorem \ref{global}, the constant $L$ is defined from several constants appearing in the estimates for the linearized equations and the constant $\epsilon$ will be chosen in such a way that $L^2\epsilon <1$. 
\\
\thetag2~ We only consider the dimension $3 \leq N \leq 7$.
 In fact, in the case $N = 2$, 
$q_1 < 2$,  and so $q_1/2 < 1$.  In this case,
our argument does not work.
Furthermore, we need a restriction $N < 8$
by the condition $q_1 \leq 4$.
\end{remark}


\section{Analysis for the linear problem}
In this section, we consider
the maximal $L_p$-$L_q$ regularity 
and decay properties of solutions,
which are the key tools for the proof of Theorem \ref{global}.

\subsection{Maximal $L_p$-$L_q$ regularity}

In this subsection, we state the maximal $L_p$-$L_q$ regularity 
for the time local linear problem:
\begin{equation}\label{l2}
\left\{
\begin{aligned}
&\pd_t \theta + \dv \bm = 0 & \quad &\text{in $\R^N$ for $t\in (0, T)$},  \\
&\pd_t \bm - \frac{1}{\rho_*} \DV \bS(\bm) -\kappa_* \rho_* \nabla \Delta \theta
= \bg & \quad &\text{in $\R^N$ for $t \in (0, T)$}, \\
&(\theta, \bm)|_{t=0} = (\rho_0, \bm_0)& \quad &\text{in $\R^N$}.
\end{aligned}
\right.
\end {equation}

If we extend $\bg$
by zero outside of $(0, T)$, by
Theorem 2.6 in \cite{Sa} and the uniquness of solutions, we have 
the following result.

\begin{thm}\label{lmr}
Let $T, R > 0$, $1 < p, q < \infty$. 
Then, there exists a constant $\delta_0 \geq 1$
such that the following assertion holds:
For any initial data $(\rho_0, \bm_0) \in D_{q, p} (\BR^N)$
with $\|(\rho_0, \bm_0)\|_{D_{q, p}(\BR^N)} \leq R$
satisfying the range condition:
\begin{equation}\label{initial}
\rho_*/2 < \rho_* + \rho_0 (x) <
2\rho_* \quad (x~{\rm \in}~\BR^N),
\end{equation}
and right member
$\bg \in L_p((0, T), L_q(\BR^N)^N)$,
problem \eqref{l2} admits 
a unique solution $(\rho, \bm) \in X_{p, q, T}$
possessing the estimate 
\begin{equation}\label{mr1}
E_{p, q}(\rho, \bm)(t)
\leq
C_{p, q, N, \delta_0, R} e^{\delta t}
\left(\|(\rho_0, \bm_0)\|_{D_{q, p} (\BR^N)}
+\|\bg\|_{L_p((0, t), L_q(\BR^N)^N)}\right)
\end{equation}
for any $t \in (0, T]$ and $\delta \geq \delta_0$, 
where we set 
\begin{equation*}
\begin{aligned}
E_{p, q}(\rho, \bu) (t)&=\|\pd_s \rho\|_{L_p((0, t), W^1_q(\BR^N))}
+
\|\rho\|_{L_p((0, t), W^3_q(\BR^N))}\\
&+
\|\pd_s \bm\|_{L_p((0, t), L_q(\BR^N)^N)}
+
\|\bm\|_{L_p((0, t), W^2_q(\BR^N)^N)},
\end{aligned}
\end{equation*}
and constant $C_{p, q, N, \delta_0, R}$ is independent of $\delta$ and $t$.

\end{thm}

\begin{remark}
Using Theorem \ref{lmr} and employing the same argument 
as in the proof of Theorem 3.1 in \cite{MS}, we also have the local well-posedness for \eqref{nsk}.
\end{remark}

\subsection{Decay property of solutions}
In this subsection, we consider the following linearized problem:
\begin{equation}\label{l1}
\left\{
\begin{aligned}
&\pd_t \theta + \dv \bm = 0 & \quad &\text{in $\R^N$ for $t \in (0, T)$},  \\
&\pd_t \bm - \alpha_* \Delta \bm - \beta_* \nabla \dv \bm -\kappa_* \rho_* \nabla \Delta \theta
= 0 & \quad &\text{in $\R^N$ for $t \in (0, T)$}, \\
&(\theta, \bm)|_{t=0} = (f, \bg)& \quad &\text{in $\R^N$},
\end{aligned}
\right.
\end {equation}
where $\alpha_* = \mu_*/\rho_*$ and $\beta_* = \nu_*/\rho_*$.
Then, by taking Fourier transform of \eqref{l1} and solving the ordinary differential equation with respect to $t$, 
$S_1(t)(f, \bg) := \theta$ and $S_2(t)(f, \bg) := \bm$ satisfy the following formula:
\begin{enumerate}
\item
If
$\delta_* := (\alpha_* + \beta_*)^2/4 - \rho_* \kappa_* \neq 0$,
we have

\begin{equation}\label{s}
\begin{aligned}
\theta
&= - \CF^{-1}_\xi
\left[ \frac{\lambda_- e^{\lambda_+ t} - \lambda_+ e^{\lambda_- t}}
{\lambda_+ - \lambda_-} \hat f \right]
- \sum^N_{k = 1} \CF^{-1}_\xi \left[ \rho_* 
\frac{e^{\lambda_+ t} - e^{\lambda_- t}}
{\lambda_+ - \lambda_-} i \xi_k \hat g_k \right],\\
\bm
&= \CF^{-1}_\xi [e^{-\alpha_* |\xi|^2 t}\hat \bg]
- \sum^N_{k = 1} 
\CF^{-1}_\xi \left[e^{-\alpha_* |\xi|^2 t} \frac{\xi \xi_k}
{|\xi|^2} \hat g_k\right]
- \CF^{-1}_\xi \left[ \kappa_* |\xi|^2
\frac{e^{\lambda_+ t} - e^{\lambda_- t}}
{\lambda_+ - \lambda_-} i\xi \hat f \right]\\
&- \sum^N_{k = 1} 
\CF^{-1}_\xi
\left[ \frac{\{(\alpha_* + \beta_*) |\xi|^2 + \lambda_-\} 
e^{\lambda_+ t} 
- \{(\alpha_* + \beta_*) |\xi|^2 + \lambda_+\}
e^{\lambda_- t}}{|\xi|^2 (\lambda_+ - \lambda_-)} 
\xi \xi_k \hat g_k \right],
\end{aligned}
\end{equation}
where 
\begin{equation}\label{lambda}
\lambda_\pm=
\left\{
\begin{aligned}
&- \displaystyle \frac{\alpha_* + \beta_*}{2}
|\xi|^2 \pm \sqrt{\delta_*}|\xi|^2 
& \delta_* > 0,\\
&- \displaystyle \frac{\alpha_* + \beta_*}{2}
|\xi|^2 \pm i \sqrt{|\delta_*|}|\xi|^2
& \delta_* < 0.
\end{aligned}
\right.
\end{equation}
\item
If $\delta_* = 0$, we have
\begin{equation}\label{s0}
\begin{aligned}
\theta
&= \CF^{-1}_\xi
\left[ e^{\lambda_0 t} (1 - \lambda t) \hat f \right]
- \sum^N_{k = 1} \CF^{-1}_\xi \left[ te^{\lambda_0 t} i\xi_k \hat g_k \right],\\
\bm
&= \CF^{-1}_\xi [e^{-\alpha_* |\xi|^2 t}\hat \bg]
- \sum^N_{k = 1} 
\CF^{-1}_\xi \left[e^{-\alpha_* |\xi|^2 t} \frac{\xi \xi_k}
{|\xi|^2} \hat g_k\right]
- \CF^{-1}_\xi \left[ e^{\lambda_0 t} \frac{t \lambda^2}{|\xi|^2}
i\xi \hat f \right]\\
& \enskip + \sum^N_{k = 1} 
\CF^{-1}_\xi
\left[ e^{\lambda_0 t} \frac{1 + t \lambda}{|\xi|^2} 
\xi \xi_k \hat g_k \right],
\end{aligned}
\end{equation}
where 
\begin{equation*}
\lambda_0=
-\frac{\alpha_* + \beta_*}{2}
|\xi|^2.
\end{equation*}
\end{enumerate}

To state decay estimates of $\theta$ and $\bm$,
we  divide the solution formula
into the low frequency part and high frequency part. For this purpose, 
we introduce a cut off function $\varphi(\xi) \in C^\infty(\BR^N)$
 which equals $1$ for $|\xi| \leq \epsilon$ and $0$ for $|\xi| \geq 2\epsilon$,
 where $\epsilon$ is a suitably small positive constant. Let $\Phi_0$ and $\Phi_\infty$ be 
 operators acting on $(f, \bg) \in W^{1,0}_q(\BR^N)$ defined by setting
 $$\Phi_0(f, \bg) = \CF^{-1}_\xi[\varphi(\xi)(\hat f(\xi), \hat\bg(\xi))],
 \quad \Phi_\infty(f, \bg) = \CF^{-1}_\xi[(1-\varphi(\xi))(\hat f(\xi), \hat\bg(\xi))].
 $$

\begin{thm}\label{semi}
Let $S_i(t)$ $(i=1,2)$ be the solution operators of \eqref{l1}
given \eqref{s} and let 
$S^0(t)(f, \bg) = (S^0_1(t) (f, \bg), S^0_2(t) (f, \bg))$ and 
$S^\infty (t)(f, \bg) = (S^\infty_1(t) (f, \bg), S^\infty_2(t) (f, \bg))$ with
$S^0_i(t)(f, \bg) = S_i(t)\Phi_0(f, \bg)$ and $S^\infty_i(t)(f, \bg) = S_i(t)\Phi_\infty(f, \bg)$.
Then, $S^0(t)$ and $S^\infty(t)$ have the following decay properties
\begin{enumerate}
\item
\begin{equation}\label{esemi1}
\|\pd^j_x S^0(t) (f, \bg)\|_{L_p(\R^N)}
\leq C
t^{-\frac{N}{2} (\frac{1}{q} - \frac{1}{p}) - \frac{j}{2}}
\|(f, \bG)\|_{L_q(\R^N)}
\end{equation}
with $\bg = \DV \bG$, $j \in \BN_0$ and 
some constant $C$ depending on 
$j$, $p$, $q$, $\alpha_*$ and $\beta_*$,
where
\begin{equation}\label{pqcondi1}
\left\{
\begin{aligned}
&1 < q \leq p \leq \infty \text{ and } 
(p, q) \neq (\infty, \infty) 
&\text{ if } 0 < t \leq 1,\\
& 1< q \leq 2 \leq p \leq \infty \text{ and } 
(p, q) \neq (\infty, \infty) 
&\text{ if } t \geq 1.
\end{aligned}
\right.
\end{equation}

\item
\begin{equation}\label{esemi2}
\|\pd^j_x S^\infty(t) (f, \bg)\|_{W^{1, 0}_p(\R^N)}
\leq C
t^{-\frac{N}{2} (\frac{1}{q} - \frac{1}{p}) - \frac{j}{2}}
\|(f, \bg)\|_{W^{1, 0}_q(\R^N)}
\end{equation}
with $j \in \BN_0$ and 
some constant $C$ depending on 
$j$, $p$, $q$, $\alpha_*$ and $\beta_*$,
where
\begin{equation}\label{pqcondi2}
1 < q \leq p \leq \infty \text{ and } 
(p, q) \neq (\infty, \infty) .
\end{equation}

\end{enumerate}
\end{thm}

\begin{proof}
First, we consider the case where $\delta_* \neq 0$.
The difference between the cases where $P'(\rho_*) = 0$ and $P'(\rho_*) > 0$ is that 
$\lambda_\pm$ satisfies \eqref{lambda} not only for the high frequency part, 
but also for the low frequency part.
By this difference, the second term of $S_1(t)$ given \eqref{s} has a trouble
because of $\lambda_+ - \lambda_- = C_* |\xi|^2$, where
$C_* = 2 \sqrt{\delta_*}$ if $\delta_* > 0$ and 
$C_* = 2 i \sqrt{|\delta_*|}$ if $\delta_* < 0$.
Due to the condition $\bg = \DV\bG$,
the second term of $S_1(t)$ satisfies
\[
- \sum^N_{k = 1} \CF^{-1}_\xi \left[ \rho_* 
\frac{e^{\lambda_+ t} - e^{\lambda_- t}}
{\lambda_+ - \lambda_-} i \xi_k \hat g_k \right]
= \sum^N_{j, k = 1} \CF^{-1}_\xi \left[ \rho_* 
\frac{e^{\lambda_+ t} - e^{\lambda_- t}}
{C_* |\xi|^2} \xi_j  \xi_k \hat G_{jk} \right]
,\]
so that we can employ the same calculation as in the proof of Theorem 4.1 in \cite{MS}. 

Next, we consider the case where $\delta_* = 0$.
Using the condition $\bg = \DV\bG$ and the estimate $(|\xi| t^{1/2})^j e^{-C_0 |\xi|^2 t} \leq C e^{-(C_0/2)|\xi|^2 t}$
for $j \in \BN_0$ with some constant $C_0$ depending on $\alpha_*$ and $\beta_*$, 
the solution formula \eqref{s0} can be estimated in the same manner as in the proof of Theorem 4.1 in \cite{MS}.  
This completes the proof of Theorem \ref{semi}.
\end{proof}

\section{A proof of Theorem \ref{global}}
We prove Theorem \ref{global} by the Banach fixed point argument.
Let $p$, $q_1$ and $q_2$ be exponents given in Theorem \ref{global}.
Let $\epsilon$ be a small positive number 
and let $\CN (\theta, \bm)$ be the norm defined in \eqref{N}.
We define the underlying space $\CI_\epsilon$ by setting
\begin{equation}\label{space}
\CI_\epsilon
= \{ (\theta, \bm) \in X_{p, \frac{q_1}{2}, \infty} \cap X_{p, q_2, \infty}
\mid (\theta, \bm)|_{t=0} = (\rho_0, \bm_0), 
\enskip \CN(\theta, \bm) (\infty) \leq L \epsilon
\}
\end{equation}
with some constant $L$ which will be determined later. 
Given $(\theta, \bm) \in \CI_\epsilon$, 
let $(\omega, \bw)$ be a solution to the equation:
\begin{equation}\label{nsk5}
\left\{
\begin{aligned}
&\pd_t \omega + \rho_* \dv \bw = 0 & \quad &\text{in $\R^N$ for $t \in (0, T)$}, \\
&\pd_t \bw - \frac{1}{\rho_*} \DV \bS(\bw) -\kappa_* \rho_* \nabla \Delta \omega 
= \bg (\theta, \bm)  & \quad & \text{in $\R^N$ for $t \in (0, T)$}, \\
&(\omega, \bw)|_{t=0} = (\rho_0, \bm_0)& \quad & \text{in $\R^N$}.
\end{aligned}
\right.
\end {equation}
We shall prove the following inequality by several steps.
\begin{equation}\label{extend}
\CN(\omega, \bw) (t) \leq C(\CI + \CN(\theta, \bm) (t)^2),
\end{equation}
where
$\CI$ is defined in Theorem \ref{global}.
Throughout the following steps, 
we use the estimate
\begin{equation}\label{infty}
\frac{\rho_*}{4} \leq
\rho_* + \theta(t, x) \leq 4 \rho_*,
\end{equation} 
which is obtained by 
$(\theta, \bm) \in X_{p, \frac{q_1}{2}, \infty} \cap X_{p, q_2, \infty}$.

\subsection{Estimates of $(\nabla^j \omega, \nabla^j \bw)$ for $j = 0, 1$}
\subsubsection{In the case that $t>2$}\label{t>2}
In order to estimate $(\omega, \bw)$ in the case that $t>2$, 
we  write $(\omega, \bw)$ by Duhamel's principle as follows:
\begin{equation}\label{duhamel}
(\omega, \bw) = S(t) (\rho_0, \bm_0) + \int^t_0 S(t - s) (0, \bg(s))\, ds.
\end{equation}
Since $S(t) (\rho_0, \bm_0)$ can be estimated directly by Theorem \ref{semi},
we only estimate the second term
for the low frequencies and the high frequencies, below.
We divide the second term into three parts as follows.
\begin{align}
\int^t_0 \|\pd_x ^j S^d (t - s) (0, \bg(s))\|_{L_X}\, ds 
= \left( \int^{t/2}_0 + \int^{t-1}_{t/2} + \int^t_{t-1}\right) \|\pd_x^j S^d(t - s)(f(s), \bg(s))\|_{L_X}\, ds
=: \sum^3_{ k= 1}I_X^{k, d}
\end{align}
for $t > 2$, where $d = 0$, $\infty$ and $X = \infty$, $q_1$, $q_2$.

\noindent
\underline{\bf  Estimates for the low frequency part in $L_\infty$}

By \eqref{infty} and Theorem \ref{semi} (i) with 
$(p, q) = (\infty, q_1/2)$
and H\"older's inequality
under the condition $q_1/2 \leq 2$, 
we have
\begin{align}\label{d1}
I_\infty^{1, 0} &\leq C \int^{t/2}_0 (t - s)^{-\frac{N}{q_1} - \frac{j}{2}} 
\|\bG\|_{L_{q_1/2}(\BR^N)} \,ds
\leq C\int^{t/2}_0 (t - s)^{-\frac{N}{q_1} - \frac{j}{2}} (A_1 + B_1) \,ds, 
\end{align}
where
\begin{align*}
A_1 &=\|(\theta, \bm)\|_{L_{q_1}(\R^N)}^2
+\|(\theta, \bm)\|_{L_{q_1}(\R^N)}
\|(\nabla \theta, \nabla \bm) \|_{L_{q_1}(\R^N)}
+\|\nabla \theta\|_{L_{q_1}(\R^N)}^2, \\
B_1&=\|\theta\|_{L_{q_1}(\R^N)}
\|\nabla^2 \theta\|_{L_{q_1}(\R^N)}.
\end{align*}
satisfying
\begin{align}
A_1 &\leq 
<s>^{-\frac{N}{q_1}}
[(\theta, \bm)]_{q_1, \frac{N}{2q_1}, t}^2
+ <s>^{-(\frac{N}{q_1}+\frac{1}{2})}
[(\theta, \bm)]_{q_1, \frac{N}{2q_1}, t}
[(\nabla \theta, \nabla \bm)]_{q_1, \frac{N}{2q_1}+\frac{1}{2}, t}
\nonumber \\
&\enskip + <s>^{-(\frac{N}{q_1}+1)}
[\nabla \theta]_{q_1, \frac{N}{2q_1}+\frac{1}{2}, t}^2
 \nonumber \\
& \leq
<s>^{-\frac{N}{q_1}}
([(\theta, \bm)]_{q_1, \frac{N}{2q_1}, t}^2
+[(\theta, \bm)]_{q_1, \frac{N}{2q_1}, t}
[(\nabla \theta, \nabla \bm)]_{q_1, \frac{N}{2q_1}+\frac{1}{2}, t}
+[\nabla \theta]_{q_1, \frac{N}{2q_1}+\frac{1}{2}, t}^2).\label{A1}\\
B_1 &\leq
<s>^{-(\frac{N}{q_1}-\tau)}[\theta]_{q_1, \frac{N}{2q_1}, t}
<s>^{\frac{N}{2q_1}-\tau}
\|\theta\|_{W^2_{q_1}(\R^N)}.\label{B1}
\end{align}
Since $1- N/q_1 < 0$ and $1 - (N/q_1 - \tau)p' < 0$
as follows from $q_1 < N$
and $\tau < N/q_2 +1/p$,
by \eqref{d1}, \eqref{A1} and \eqref{B1}, we have
\begin{align}\label{infty1}
I_\infty^{1, 0}
 &\leq C t^{-\frac{N}{q_1} - \frac{j}{2}} 
 \int^{t/2}_0 <s>^{-\frac{N}{q_1}}\,ds 
 ([(\theta, \bm)]_{q_1, \frac{N}{2q_1}, t}^2
+[(\theta, \bm)]_{q_1, \frac{N}{2q_1}, t}
[(\nabla \theta, \nabla\bm)]_{q_1, \frac{N}{2q_1}+\frac{1}{2}, t}
+[\nabla \theta]_{q_1, \frac{N}{2q_1}+\frac{1}{2}, t}^2)
 \nonumber\\ 
& \enskip
+ C t^{-\frac{N}{q_1} - \frac{j}{2}} \left(\int^{t/2}_0 
<s>^{-(\frac{N}{q_1} -\tau)p'}\,ds\right)^{1/p'} 
[\theta]_{q_1, \frac{N}{2q_1}, t}
\|<s>^{\frac{N}{2q_1}-\tau} \theta\|_{L_p((0, t), W^2_{q_1}(\R^N))}
\nonumber\\
&\leq C t^{-\frac{N}{q_1} - \frac{j}{2}} E_0^0 (t),
\end{align}
where 
\begin{align*}
E_0^0(t) &= [(\theta, \bm)]_{q_1, \frac{N}{2q_1}, t}^2
+[(\theta, \bm)]_{q_1, \frac{N}{2q_1}, t}
[(\nabla \theta, \nabla \bm)]_{q_1, \frac{N}{2q_1}+\frac{1}{2}, t}
+[\nabla \theta]_{q_1, \frac{N}{2q_1}+\frac{1}{2}, t}^2\\
& \enskip 
+[\theta]_{q_1, \frac{N}{2q_1}, t}
\|<s>^{\frac{N}{2q_1}-\tau} \theta\|_{L_p((0, t), W^2_{q_1}(\R^N))}.
\end{align*}
Analogously, we have
\begin{equation}\label{infty2}
I_\infty^{2, 0} \leq C t^{-\frac{N}{q_1} - \frac{j}{2}} E_0^0(t).
\end{equation}

We now estimate $I^{3, 0}_\infty$. 
By \eqref{infty} and Theorem \ref{semi} (i) with 
$(p, q) = (\infty, q_2)$,
we have
\begin{align}\label{d2}
I_\infty^{3, 0} &\leq C \int^t_{t-1} (t - s)^{-\frac{N}{2q_2} - \frac{j}{2}} 
\|\bG\|_{L_{q_2}(\BR^N)} \,ds
\leq C\int^t_{t-1} (t - s)^{-\frac{N}{2q_2} - \frac{j}{2}} (A_2 + B_2) \,ds, 
\end{align}
where
\begin{align*}
A_2 &=\|(\theta, \bm)\|_{L_ \infty (\R^N)}
(\|(\theta, \bm)\|_{L_{q_2} (\R^N)} 
+ \|(\nabla \theta, \nabla \bm)\|_{L_{q_2} (\R^N)})
+\|\nabla \theta\|_{L_\infty (\R^N)} \|\nabla \theta\|_{L_{q_2}(\R^N)}, \\
B_2&=\|\theta\|_{L_\infty (\R^N)}
\|\nabla^2 \theta\|_{L_{q_2}(\R^N)}
\end{align*}
satisfying
\begin{align}
A_2 &\leq <s>^{-(\frac{N}{q_1}+\frac{N}{2q_2}+1)} 
[(\theta, \bm)]_{\infty, \frac{N}{q_1}, t}
[(\theta, \bm)]_{q_2, \frac{N}{2q_2}+1, t}\nonumber \\
&\enskip +<s>^{-(\frac{N}{q_1}+\frac{N}{2q_2}+\frac{3}{2})}
[(\theta, \bm)]_{\infty, \frac{N}{q_1}, t}
[(\nabla \theta, \nabla \bm)]_{q_2, \frac{N}{2q_2}+\frac{3}{2}, t}\nonumber \\
& \enskip + <s>^{-(\frac{N}{q_1}+\frac{N}{2q_2}+2)}
[\nabla \theta]_{\infty, \frac{N}{q_1}+\frac{1}{2}, t}
[(\nabla \theta, \nabla \bm)]_{q_2, \frac{N}{2q_2}+\frac{3}{2}, t},\label{A2} \\
B_2 &\leq
<s>^{-(\frac{N}{q_1}+\frac{N}{2q_2}+1-\tau)}
[\theta]_{\infty, \frac{N}{q_1}, t}
<s>^{\frac{N}{2q_2}+1-\tau}
\|\theta\|_{W^2_{q_2}(\R^N)}.\label{B2}
\end{align}
Since $1-(N/2q_2 + j/2) > 0$, $1 - (N/2q_2 + j/2)p' > 0$,
and $N/2q_2 + 1 -\tau > j/2$ as follows 
from $N < q_2$, $2/p + N/q_2<1$ and
$\tau < N/q_2 + 1/p$,
by \eqref{d2}, \eqref{A2} and \eqref{B2}, we have
\begin{align}\label{infty3}
I_\infty^{3, 0}
&\leq 
C t^{-(\frac{N}{q_1}+\frac{N}{2q_2}+1)} 
\int^t_{t-1} (t - s)^{- ( \frac{N}{2q_2} + \frac{j}{2} )}\,ds 
[(\theta, \bm)]_{\infty, \frac{N}{q_1}, t}
[(\theta, \bm)]_{q_2, \frac{N}{2q_2}+1, t} \nonumber \\
& \enskip 
+ 
C t^{-(\frac{N}{q_1}+\frac{N}{2q_2}+\frac{3}{2})} 
\int^t_{t-1} (t - s)^{- ( \frac{N}{2q_2} + \frac{j}{2} )}\,ds 
[(\theta, \bm)]_{\infty, \frac{N}{q_1}, t}
[(\nabla \theta, \nabla \bm)]_{q_2, \frac{N}{2q_2}+\frac{3}{2}, t}\nonumber \\
& \enskip
+
C t^{-(\frac{N}{q_1}+\frac{N}{2q_2}+2)} 
\int^t_{t-1} (t - s)^{- ( \frac{N}{2q_2} + \frac{j}{2} )}\,ds 
[\nabla \theta]_{\infty, \frac{N}{q_1}+\frac{1}{2}, t}
[(\nabla \theta, \nabla \bm)]_{q_2, \frac{N}{2q_2}+\frac{3}{2}, t}\nonumber \\
& \enskip
+
Ct^{-(\frac{N}{q_1}+\frac{N}{2q_2}+1-\tau)}
\left(\int^t_{t-1} (t - s)^{- ( \frac{N}{2q_2} + \frac{j}{2} )p'}\,ds\right)^{1/p'} 
[\theta]_{\infty, \frac{N}{q_1}, t}
\|<s>^{\frac{N}{2q_2}+1-\tau}
\theta\|_{L_p((0, t), W^2_{q_2}(\R^N))}\nonumber\\
&\leq C t^{-\frac{N}{q_1} - \frac{j}{2}} E_2^0(t),
\end{align}
where 
\begin{align*}
E_2^0(t) &= [(\theta, \bm)]_{\infty, \frac{N}{q_1}, t}
\{[(\theta, \bm)]_{q_2, \frac{N}{2q_2}+1, t}
+[(\nabla \theta, \nabla \bm)]_{q_2, \frac{N}{2q_2}+\frac{3}{2}, t}\}\\
&\enskip 
+[\nabla \theta]_{\infty, \frac{N}{q_1}+\frac{1}{2}, t}
[(\nabla \theta, \nabla \bm)]_{q_2, \frac{N}{2q_2}+\frac{3}{2}, t}
+[\theta]_{\infty, \frac{N}{q_1}, t}
\|<s>^{\frac{N}{2q_2}+1-\tau} \theta\|_{L_p((0, t), W^2_{q_2}(\R^N))}.
\end{align*}

By \eqref{infty1}, \eqref{infty2} and \eqref{infty3}, we have
\begin{equation}\label{infty4}
\int^t_0 \|\pd_x ^j S^0 (t - s) (0, \bg(s))\|_{L_\infty}\, ds 
\leq C t^{-\frac{N}{q_1} - \frac{j}{2}} (E_0^0 (t) + E_2^0 (t)).
\end{equation}

\noindent
\underline{\bf Estimates for the low frequency part in $L_{q_1}$}

Using \eqref{infty} and Theorem \ref{semi} (i) with $(p, q) = (q_1, q_1/2)$ and employing 
the same calculation as in the estimate in $L_\infty$, we have
\begin{equation}\label{q11}
I_{q_1}^{1, 0} + I_{q_1}^{2, 0} \leq C t^{-\frac{N}{2q_1} - \frac{j}{2}} E_0^0(t).
\end{equation}
By Theorem \ref{semi} (i) with 
$(p, q) = (q_1, q_1)$,
we have
\begin{align}\label{d3}
I_{q_1}^{3, 0} &\leq C \int^t_{t-1} (t - s)^{- \frac{j}{2}} 
\|\bG\|_{L_{q_1}(\BR^N)} \,ds
\leq C\int^t_{t-1} (t - s)^{- \frac{j}{2}} (A_3 + B_3) \,ds, 
\end{align}
where
\begin{align*}
A_3 &=\|(\theta, \bm)\|_{L_ \infty (\R^N)}
(\|(\theta, \bm)\|_{L_{q_1} (\R^N)} 
+ \|(\nabla \theta, \nabla \bm)\|_{L_{q_1} (\R^N)})
+\|\nabla \theta\|_{L_\infty (\R^N)} \|\nabla \theta\|_{L_{q_1}(\R^N)}, \\
B_3&=\|\theta\|_{L_\infty (\R^N)}
\|\nabla^2 \theta\|_{L_{q_1}(\R^N)}
\end{align*}
satisfying
\begin{align}
A_3 &\leq <s>^{-\frac{3N}{2q_1}} 
[(\theta, \bm)]_{\infty, \frac{N}{q_1}, t}
[(\theta, \bm)]_{q_1, \frac{N}{2q_1}, t}\nonumber \\
&\enskip +<s>^{-(\frac{3N}{2q_1}+\frac{1}{2})}
[(\theta, \bm)]_{\infty, \frac{N}{q_1}, t}
[(\nabla \theta, \nabla \bm)]_{q_1, \frac{N}{2q_1}+\frac{1}{2}, t}\nonumber \\
& \enskip + <s>^{-(\frac{3N}{2q_1}+1)}
[\nabla \theta]_{\infty, \frac{N}{q_1}+\frac{1}{2}, t}
[(\nabla \theta, \nabla \bm)]_{q_1, \frac{N}{2q_1}+\frac{1}{2}, t},\label{A3} \\
B_3 &\leq
<s>^{-(\frac{3N}{2q_1}-\tau)}
[\theta]_{\infty, \frac{N}{q_1}, t}
<s>^{\frac{N}{2q_1}-\tau}
\|\theta\|_{W^2_{q_1}(\R^N)}.\label{B3}
\end{align}
Since $1 - (j/2)p' > 0$,
and $3N/2q_1 - \tau >N/2q_1 + j/2$ as follows 
from $p > 2$ and
$\tau < N/q_2 + 1/p$,
by \eqref{d3}, \eqref{A3} and \eqref{B3}, we have
\begin{align}\label{q13}
I_{q_1}^{3, 0}
&\leq 
C t^{-\frac{3N}{2q_1}} 
\int^t_{t-1} (t - s)^{- \frac{j}{2}}\,ds 
[(\theta, \bm)]_{\infty, \frac{N}{q_1}, t}
[(\theta, \bm)]_{q_1, \frac{N}{2q_1}, t}\nonumber \\
& \enskip 
+ 
C t^{-(\frac{3N}{2q_1}+\frac{1}{2})}
\int^t_{t-1} (t - s)^{- \frac{j}{2}}\,ds 
[(\theta, \bm)]_{\infty, \frac{N}{q_1}, t}
[(\nabla \theta, \nabla \bm)]_{q_1, \frac{N}{2q_1}+\frac{1}{2}, t}\nonumber \\
& \enskip
+
C t^{-(\frac{3N}{2q_1}+1)} 
\int^t_{t-1} (t - s)^{-\frac{j}{2}}\,ds 
[\nabla \theta]_{\infty, \frac{N}{q_1}+\frac{1}{2}, t}
[(\nabla \theta, \nabla \bm)]_{q_1, \frac{N}{2q_1}+\frac{1}{2}, t}\nonumber \\
& \enskip
+
Ct^{-(\frac{3N}{2q_1}-\tau)}
\left(\int^t_{t-1} (t - s)^{- \frac{j}{2} p'}\,ds\right)^{1/p'} 
[\theta]_{\infty, \frac{N}{q_1}, t}
\|<s>^{\frac{N}{2q_1}-\tau}
\theta\|_{L_p((0, t), W^2_{q_1}(\R^N))}\nonumber\\
&\leq C t^{-\frac{N}{2q_1} - \frac{j}{2}} E_1^0(t),
\end{align}
where 
\begin{align*}
E_1^0(t)
&=[(\theta, \bm)]_{\infty, \frac{N}{q_1}, t}
\{[(\theta, \bm)]_{q_1, \frac{N}{2q_1}, t}
+[(\nabla \theta, \nabla \bm)]_{q_1, \frac{N}{2q_1}+\frac{1}{2}, t}\}\\
&\enskip 
+[\nabla \theta]_{\infty, \frac{N}{q_1}+\frac{1}{2}, t}
[(\nabla \theta, \nabla \bm)]_{q_1, \frac{N}{2q_1}+\frac{1}{2}, t}
+[\theta]_{\infty, \frac{N}{q_1}, t}
\|<s>^{\frac{3N}{2q_1}-\tau} \theta\|_{L_p((0, t), W^2_{q_1}(\R^N))}.
\end{align*}
By \eqref{q11} and \eqref{q13}, we have
\begin{equation}\label{q1}
\int^t_0 \|\pd_x ^j S^0 (t - s) (0, \bg(s))\|_{L_{q_1}}\, ds 
\leq C t^{-\frac{N}{2q_1} - \frac{j}{2}} (E_0^0 (t) + E_1^0 (t)).
\end{equation}

\noindent
\underline{\bf Estimates for the low frequency part in $L_{q_2}$}

Using \eqref{infty} and Theorem \ref{semi} (i) with $(p, q) = (q_2, q_1/2)$ and $(p, q) = (q_2, q_2)$, we have
\begin{equation}\label{q2}
\int^t_0 \|\pd_x ^j S^0 (t - s) (0, \bg(s))\|_{L_{q_2}}\, ds 
\leq C t^{-\frac{N}{2q_2} - 1 - \frac{j}{2}} (E_0^0 (t) + E_2^0 (t)).
\end{equation}

\noindent
\underline{\bf Estimates for the high frequency part}

Employing the same calculation as in Step1, 
we have estimates for the high frequencies under the conditions \eqref{condi:pq} and \eqref{tau} as follows.
\begin{align}\label{highfre}
&\int^t_0 \|\pd_x ^j S^\infty (t - s) (0, \bg(s))\|_{L_\infty}\, ds 
\leq C t^{-\frac{N}{q_1} - \frac{j}{2}} (E_0^\infty (t) + E_2^\infty (t)),\nonumber \\
&\int^t_0 \|\pd_x ^j S^\infty (t - s) (0, \bg(s))\|_{L_{q_1}}\, ds 
\leq C t^{-\frac{N}{2q_1} - \frac{j}{2}} (E_0^\infty (t) + E_1^\infty (t)),\\
&\int^t_0 \|\pd_x ^j S^\infty (t - s) (0, \bg(s))\|_{L_{q_2}}\, ds 
\leq C t^{-\frac{N}{2q_2} - 1 - \frac{j}{2}} (E_0^\infty (t) + E_2^\infty (t)),\nonumber 
\end{align}
where
\begin{align*}
E_0^\infty(t) &= [(\theta, \bm)]_{q_1, \frac{N}{2q_1}, t} [(\nabla \theta, \nabla \bm)]_{q_1, \frac{N}{2q_1}+\frac{1}{2}, t}\\
& \enskip
+[\nabla \theta]_{\infty, \frac{N}{q_1}+\frac{1}{2}, t} [\bm]_{q_1, \frac{N}{2q_1}, t}^2
+[\nabla \theta]_{q_1, \frac{N}{2q_1}+\frac{1}{2}, t} [\nabla \bm]_{q_1, \frac{N}{2q_1}+\frac{1}{2},t}\\
& \enskip
+[\theta]_{q_1, \frac{N}{2q_1}, t} \|<s>^{\frac{N}{2q_1}-\tau} \theta\|_{L_p((0, t), W^3_{q_1}(\BR^N))}\\
& \enskip
+([\bm]_{q_1, \frac{N}{2q_1}, t} + [\nabla \theta]_{q_1, \frac{N}{2q_1}+\frac{1}{2}, t}) 
\|<s>^{\frac{N}{2q_1}-\tau} \theta\|_{L_p((0, t), W^2_{q_1}(\BR^N))},\\
E_1^\infty(t) &= [(\theta, \bm)]_{\infty, \frac{N}{q_1}, t} [(\nabla \theta, \nabla \bm)]_{q_1, \frac{N}{2q_1}+\frac{1}{2}, t}\\
& \enskip
+[\nabla \theta]_{q_1, \frac{N}{2q_1}+\frac{1}{2}, t} [\bm]_{\infty, \frac{N}{q_1}, t}^2
+[\nabla \theta]_{\infty, \frac{N}{q_1}+\frac{1}{2}, t} [\nabla \bm]_{q_1, \frac{N}{2q_1}+\frac{1}{2},t}\\
& \enskip
+[\theta]_{\infty, \frac{N}{q_1}, t} \|<s>^{\frac{N}{2q_1}-\tau} \theta\|_{L_p((0, t), W^3_{q_1}(\BR^N))}\\
& \enskip
+([\bm]_{\infty, \frac{N}{q_1}, t} + [\nabla \theta]_{\infty, \frac{N}{q_1}+\frac{1}{2}, t}) 
\|<s>^{\frac{N}{2q_1}-\tau} \theta\|_{L_p((0, t), W^2_{q_1}(\BR^N))},\\
E_2^\infty(t) &= [(\theta, \bm)]_{\infty, \frac{N}{q_1}, t} [(\nabla \theta, \nabla \bm)]_{q_2, \frac{N}{2q_2}+\frac{3}{2}, t}\\
& \enskip
+[\nabla \theta]_{q_2, \frac{N}{2q_2}+\frac{3}{2}, t} [\bm]_{\infty, \frac{N}{q_1}, t}^2
+[\nabla \theta]_{\infty, \frac{N}{q_1}+\frac{1}{2}, t} [\nabla \bm]_{q_2, \frac{N}{2q_2}+\frac{3}{2},t}\\
& \enskip
+[\theta]_{\infty, \frac{N}{q_1}, t} \|<s>^{\frac{N}{2q_2}+1-\tau} \theta\|_{L_p((0, t), W^3_{q_2}(\BR^N))}\\
& \enskip
+([\bm]_{\infty, \frac{N}{q_1}, t} + [\nabla \theta]_{\infty, \frac{N}{q_1}+\frac{1}{2}, t}) 
\|<s>^{\frac{N}{2q_2}+1-\tau} \theta\|_{L_p((0, t), W^2_{q_2}(\BR^N))}.
\end{align*}

By \eqref{duhamel}, \eqref{infty4}, \eqref{q1}, \eqref{q2} and \eqref{highfre},
we have
\begin{align}\label{2t}
&\sum^1_{j=0} [(\nabla^j \omega, \nabla^j \bw)]_{\infty, \frac{N}{q_1}+\frac{j}{2}, (2, t)}
\leq C (\|(\rho_0, \bm_0)\|_{W^{1, 0}_\frac{q_1}{2}} + \|(\rho_0, \bM_0)\|_{L_\frac{q_1}{2}}
+E_0(t) + E_2(t)), \nonumber\\
&\sum^1_{j=0} [(\nabla^j \omega, \nabla^j \bw)]_{q_1, \frac{N}{2q_1}+\frac{j}{2}, (2, t)}
\leq C (\|(\rho_0, \bm_0)\|_{W^{1, 0}_\frac{q_1}{2}} + \|(\rho_0, \bM_0)\|_{L_\frac{q_1}{2}}
+E_0(t) + E_1(t)),\\
&\sum^1_{j=0} [(\nabla^j \omega, \nabla^j \bw)]_{q_2, \frac{N}{2q_2}+1+\frac{j}{2}, (2, t)}
\leq C (\|(\rho_0, \bm_0)\|_{W^{1, 0}_\frac{q_1}{2}} + \|(\rho_0, \bM_0)\|_{L_\frac{q_1}{2}}
+E_0(t) + E_2(t)), \nonumber
\end{align}
where
$E_i(t) = E_i^0(t) + E_i^\infty(t)$ with $i = 0, 1, 2$.

\subsubsection{In the case that $0<t<2$}

\noindent
\underline{\bf Estimates in $L_{q_i}$ for $i=1, 2$}

Using Theorem \ref{lmr} and the following estimate 
\[
\|\bg\|_{L_p((0, t), L_{q_i}(\BR^N))} \leq C E_i^\infty (t),
\]
which is calculated
in \ref{t>2},
so that we have
\begin{align}\label{mr2}
&\|(\omega, \bw)\|_{L_p((0, 2), W^{3, 2}_{q_i}(\BR^N))} 
+ \|(\pd_s \omega, \pd_s \bw)\|_{L_p((0, 2), W^{1, 0}_{q_i}(\BR^N))}\nonumber \\
&\leq C\{ \|(\rho_0, \bm_0)\|_{D_{q_i, p} (\BR^N)} + E_i (2) \}
\end{align}
for $i = 1, 2$. 

\noindent
\underline{\bf Estimates in $L_\infty$}

In order to estimate in $L_\infty$, we use the following Lemma. 
(cf. Lemma 1 in \cite{SS0} and Lemma 3.3 in \cite{MS})

\begin{lem}\label{sup}
Let $\bu \in W^1_p((0, T), L_q(\BR^N)^N) \cap L_p((0, T), W^2_q(\BR^N)^N)$
and
$\rho \in W^1_p((0, T), W^1_q(\BR^N)) \cap L_p((0, T), W^3_q(\BR^N))$,
with $2 < p < \infty$, $1 < q < \infty$ and $T > 0$.
Then,
\begin{equation}\label{supq}
\sup_{0 < s <T} \|(\rho (\cdot, s), \bu (\cdot, s))\|_{D_{q, p}(\BR^N)}
\leq C\{\|(\rho (\cdot, 0), \bu (\cdot, 0))\|_{D_{q, p}(\BR^N)} + E_{p, q} (\rho, \bu) (T)\}
\end{equation}
with the constant $C$ independent of $T$.

If we assume that
$2/p + N/q <1$ in addition, then
\begin{equation}\label{supinfty}
\sup_{0 < s <S} \|(\rho (\cdot, s), \bu (\cdot, s))\|_{W^{2, 1}_\infty(\BR^N)}
\leq C\{\|(\rho (\cdot, 0), \bu (\cdot, 0))\|_{D_{q, p}(\BR^N)} + E_{p, q} (\rho, \bu) (S)\}
\end{equation}
for any $S \in (0, T)$ with the constant $C$ independent of $S$ and $T$.
\end{lem}

By Lemma \ref{sup}, we have 
\begin{align}\label{embedd}
\|(\omega, \bw)\|_{L_\infty((0, 2), W^1_\infty(\BR^N))}
&\leq C \{\|(\rho_0, \bm_0)\|_{D_{q_2, p} (\BR^N)} + E_2(2)\}.
\end{align}

\subsubsection{Conclusion}

Combining \eqref{2t}, \eqref{mr2} and \eqref{embedd},
we have
\begin{align}\label{low}
&\sum^1_{j = 0}[(\nabla^j \omega, \nabla^j \bw)]_{\infty, \frac{N}{q_1} + \frac{j}{2}, (0, t)} 
\leq C (\CI + E_0 (t) + E_2 (t)), \nonumber \\
&\sum^1_{j = 0}[(\nabla^j \omega, \nabla^j \bw)]_{q_1, \frac{N}{2q_1} + \frac{j}{2}, (0, t)} 
\leq C (\CI + E_0 (t) + E_1 (t)),\\
&\sum^1_{j = 0}[(\nabla^j \omega, \nabla^j \bw)]_{q_2, \frac{N}{2q_2} + 1 + \frac{j}{2}, (0, t)} 
\leq C (\CI + E_0 (t) + E_2 (t)).\nonumber
\end{align}

\subsection{Estimates of the weighted norm in the maximal $L_p$-$L_q$ regularity class}
In order to estimate the weighted norm 
in the maximal $L_p$-$L_q$ regularity class,
we consider the following time shifted equations, 
which is equivalent to the first and the second equations of \eqref{nsk5}:
\begin{align*}
&\pd_s ( <s>^{\ell_i} \omega) 
+ \delta_0 <s>^{\ell_i} \omega 
+ \rho_* \dv (<s>^{\ell_i} \bw)\\ 
&= \delta_0 <s>^{\ell_i} \omega 
+ (\pd_s <s>^{\ell_i}) \omega \\
&\pd_s (<s>^{\ell_i} \bw) 
+ \delta_0<s>^{\ell_i} \bw 
- \alpha_* \Delta (<s>^{\ell_i} \bw) 
- \beta_* \nabla (\dv <s>^{\ell_i} \bw) \\
& \enskip + \kappa_* \nabla \Delta <s>^{\ell_i} \omega 
- \gamma_* \nabla <s>^{\ell_i} \omega \\
&= <s>^{\ell_i} \bg (\theta, \bm)  
+ \delta_0 <s>^{\ell_i} \bw 
+ (\pd_s <s>^{\ell_i})\bw,
\end {align*}
where $i=1, 2$, $\ell_1 = N/2q_1 - \tau$ and $\ell_2 = N/2q_2 + 1 - \tau$.
By Theorem \ref{lmr}, we have
\begin{align}\label{high}
&\| <s>^{\ell_i} (\omega, \bw)\|_{L_p((0, t), W^{3, 2}_{q_i}(\BR^N))}
+ \| <s>^{\ell_i} (\pd_s \omega, \pd_s \bw)\|_{L_p((0, t), W^{1, 0}_{q_i}(\BR^N))}\nonumber \\
&\leq C (\|(\rho_0, \bm_0)\|_{D_{q_i, p} (\BR^N)}
+ \|<s>^{\ell_i} \bg (\theta, \bm) \|_{L_p((0, t), L_{q_i}(\BR^N))} \nonumber\\
&+ \|<s>^{\ell_i} (\omega, \bw)\|_{L_p((0, t), W^{1, 0}_{q_i}(\BR^N))}
+ \|(\pd_s <s>^{\ell_i})(\omega, \bw)\|_{L_p((0, t), W^{1, 0}_{q_i}(\BR^N))}.
\end{align}
We can estimate the left-hand sides 
of \eqref{high} by the same calculation
as in \cite{MS}, we have
\begin{align}\label{high1}
&\| <s>^{\ell_i} (\omega, \bw)\|_{L_p((0, t), W^{3, 2}_{q_i}(\BR^N))}
+ \| <s>^{\ell_i} (\pd_s \omega, \pd_s \bw)\|_{L_p((0, t), W^{1, 0}_{q_i}(\BR^N))}\nonumber \\
&\leq C (\CI + E_0(t) + E_i (t)).
\end{align}

\subsection{Conclusion}
Combining \eqref{low} and \eqref{high1}, we have
\eqref{extend}.
Recalling that $\CI \leq \epsilon$, for $(\theta, \bm) \in \CI_\epsilon$,
we have
\begin{equation}\label{extend*}
\CN(\omega, \bw)(\infty) \leq C(\CI + \CN(\theta, \bm)(\infty)^2)
\leq C\epsilon + CL^2\epsilon^2.
\end{equation}
Choosing $\epsilon$ so small that  
$L^2 \epsilon \leq 1$ 
and setting $L = 2C$
in \eqref{extend*},
we have 
\begin{equation}\label{est:global}
\CN(\omega, \bw)(\infty) \leq  L\epsilon.
\end{equation}
We define a map $\Phi$ acting on $(\theta, \bm) 
\in \CI_\epsilon$ by $\Phi(\theta, \bm) = (\omega, \bw)$, 
and then it follows from \eqref{est:global}
that $\Phi$ is the map from $\CI_\epsilon$ into itself.
Considering the difference
$\Phi(\theta_1, \bm_1) - \Phi(\theta_2, \bm_2)$
for $(\theta_i, \bm_i) \in \CI_\epsilon$ $(i = 1, 2)$,
employing the same argument as in the proof of \eqref{extend*}
and choosing $\epsilon > 0$
samller if necessary,
we see that $\Phi$ is a consraction map on $\CI_\epsilon$,
and therefore there exists a fixed point $(\omega, \bw) \in \CI_\epsilon$
which solves the equation \eqref{nsk4}.
Since the existence of solutions to \eqref{nsk4} is proved 
by the contraction mapping principle,
the uniqueness of solutions belonging to $\CI_\epsilon$ follows immediately,
which completes the proof of Theorem \ref{global}.


\begin{thebibliography}{9}
{\small 

\bibitem{BYZ} D.~Bian, L.~Yao, and C.~Zhu, 
 {\it Vanishing capillarity limit of the compressible 
fluid models of
Korteweg type to the Navier-Stokes equations}, 
SIAM J. Math. Anal.,
{\bf 46 (2)} (2014) 1633--1650.

\bibitem{Bourgain} J.~Bourgain, {\it Vector-valued singular 
integrals and the $H^1$-BMO duality}, In: Probability Theory and 
Harmonic Analysis, D.~Borkholder (ed.) {\it Marcel Dekker, 
New York} (1986) 1--19. 


\bibitem{BDL} D.~Bresch, B.~Desjardins and C.~K.~Lin, 
 {\it On some compressible fluid models: Korteweg, lubrication
and shallow water systems}, 
Comm. Partial Differential Equations,
{\bf 28} (2003) 843--868.

\bibitem{CZ} Z.~Chen and H.~Zhao, 
 {\it Existence and nonlinear stability of stationary solutions to the full com-
pressible Navier-Stokes-Korteweg system}, 
J. Math. Pures Appl. (9), {\bf 101(3)}
(2014) 330--371.

\bibitem{CK}N.~Chikami and T.~Kobayashi,
{\it Global well-posedness and time-decay estimates of the compressible Navier-Stokes-Korteweg system in critical Besov spaces},
 J. Math. Fluid Mech. {\bf 21} (2019), no. 2, Art. 31.


\bibitem{DD} R.~Danchin, B.~Desjardins, 
 {\it Existence of solutions for compressible fluid models of Korteweg
type}, 
Ann. Inst. Henri Poincare Anal. Nonlinear,
{\bf 18} (2001) 97--133.

\bibitem{D}J.~Daube, 
{\it Sharp-Interface Limit for the Navier-Stokes-Korteweg Equations}, 
Doktorarbeit, Universitat Freiburg, 2017.

\bibitem{DS} J.~E.~Dunn and J.~Serrin, 
 {\it On the thermomechanics of interstital working}, 
Arch. Ration. Mech. Anal.,
{\bf 88} (1985) 95--133.

\bibitem{H} B.~Haspot, 
 {\it Existence of global weak solution for compressible fluid models of Korteweg type}, 
J. Math. Fluid Mech.,
{\bf 13} (2011) 223--249.

\bibitem{HL1} H.~Hattori, D.~Li,
 {\it Solutions for two dimensional systems for materials of Korteweg type}, SIAM
J. Math. Anal., {\bf 25} (1994) 85--98.
\bibitem{HL2} H.~Hattori, D.~Li,
 {\it Golobal solutions of a high dimensional systems for Korteweg materials}, 
J. Math. Anal. Appl., {\bf 198} (1996) 84--97.
\bibitem{HPZ} X.~Hou, H.~Peng and C.~Zhu,
 {\it Global classical solutions to the 3D Navier-Stokes-Korteweg
equations with small initial energy},
Anal. Appl., {\bf 16 (1)} (2018) 55--84.



\bibitem{KT} T.~Kobayashi and K.~Tsuda,
 {\it 
Global existence and time decay estimate of solutions to the compressible Navier-StokesKorteweg system under critical condition}.
To appear in Asymptotic
Analysis, 2020, arXiv:1905.03542.
\bibitem{K} D.~J.~Korteweg,
 {\it Sur la forme que prennent les \'equations du mouvement des fluides si lfon tient compte des forces capillaires caus\'ees par des variations de densit\'e consid\'erables mais continues et sur la th\'eorie de la capillarite dans lfhypoth\'ese dfune variation continue de la densit\'e},
Archives N\'eerlandaises des sciences exactes et naturelles, (1901) 1--24. 


\bibitem{L} Y.~P.~Li, 
 {\it Global existence and optimal decay rate for the compressible Navier-Stokes-Korteweg
equations with external force}, 
J. Math. Anal. Appl., {\bf 388} (2012) 1218--1232.
\bibitem{MS} M.~Murata and Y.~Shibata,
The global well-posedness for the compressible 
fluid model of Korteweg type,
arXiv:1908.07224, 2019.

\bibitem{Sa} H.~Saito,
{\it On the maximal $L_p$-$L_q$ regularity
for a compressible fluid model
of Korteweg type on general domains},
J.Differential Equations, {\bf 268 (6)} (2020) 2802--2851.



\bibitem{SS0} M.~Schonbek and Y.~Shibata, 
{\it On the global well-posedness of strong dynamics of incompressible nematic liquid crystals in 
$\BR^N$},
J. Evol. Equ. {\bf 17 (1)} (2017) 537--550.





\bibitem{TW} Z.~Tan, H.~Q.~Wang,
{\it Large time behavior of solutions to the isentropic 
compressible fluid
models of Korteweg type in $\BR^3$}, 
Commun. Math. Sci.
{\bf 10 (4)} (2012) 1207--1223.
\bibitem{TWX} Z.~Tan, H.~Q.~Wang, and J.~K.~Xu,
{\it Global existence and optimal $L^2$ decay rate for the strong
solutions to the compressible 
fluid models of Korteweg type}, 
J. Math. Anal. Appl.,
{\bf 390} (2012) 181--187.

\bibitem{TZ} Z.~Tan and R.~Zhang,
{\it Optimal decay rates of the compressible 
fluid models of Korteweg type}, 
Z. Angew. Math. Phys.
{\bf 65} (2014) 279--300.
\bibitem{Wa}  J.~D.~Van der Waals,
{\it Th\'eorie thermodynamique de la capillarit\'e, dans lfhypoth\'ese dfune variation continue de la densit\'e. }, 
Archives N\'eerlandaises des sciences exactes et naturelles
{\bf XXVIII} (1893) 121--209.
\bibitem{WT} Y.~J.~Wang, Z.~Tan, {\it Optimal decay rates for the compressible fluid models of Korteweg type}, 
J. Math. Anal. Appl., {\bf 379} (2011) 256--271.
\bibitem{W} K.~Watanabe, 
Global existence of the Navier-Stokes-Korteweg equations with a
non-decreasing pressure in $L^p$
- framework,
arXiv:1907.07752, 2019.
}

\end{thebibliography}
\end{document}